\documentclass[10pt]{article}
\usepackage[english]{babel}
\usepackage{amssymb,amsthm,amsmath,amsfonts,mathrsfs}
\usepackage{pst-all,pst-3dplot,pstricks,pstricks-add,pst-math,pst-xkey}
\usepackage{graphicx}
\usepackage{latexsym}
\usepackage{ifthen, bigstrut}
\usepackage{accents}
\usepackage[english]{babel}
\usepackage{hyperref}
\usepackage{breakurl}
\usepackage{stmaryrd}
\usepackage{ifthen, bigstrut, amssymb, amsthm, amsmath}

\usepackage{bbold}
\usepackage{enumerate}
\usepackage{epsfig}
\usepackage{subfigure}
\usepackage{skak}
\usepackage{chessboard}
\usepackage{chessfss}
\usepackage{pst-all}
\usepackage{lscape}
\usepackage{courier}
\usepackage[all,cmtip]{xy}

\newtheorem{theorem}{Theorem}

\newtheorem{corollary}[theorem]{Corollary}
\newtheorem{lemma}[theorem]{Lemma}

\newcommand{\N}{\mathcal{N}}
\renewcommand{\P}{\mathcal{P}}
\renewcommand{\L}{\mathcal{L}}
\newcommand{\R}{\mathcal{R}}

\newcommand{\U}{\mathcal{U}}

\newcommand{\GL}{{G^\mathcal{L}}}
\newcommand{\GR}{{G^\mathcal{R}}}

\renewcommand{\ge}{\geqslant}

\newcommand{\su}{\succcurlyeq}
\newcommand{\pr}{\preccurlyeq}

\theoremstyle{definition}
\newtheorem{definition}[theorem]{Definition}

\newtheorem{observation}[theorem]{Observation}

\begin{document}

\title{{\sc invertible elements of the\\ mis\`{e}re dicotic universe}}
\maketitle

\vspace{-0.4cm}
\begin{center}
  {\sc Michael Fisher}\textsuperscript{1},
    {\sc Richard Nowakowski}\textsuperscript{2},\\
      {\sc Carlos Pereira dos Santos}\textsuperscript{3}
  \par \bigskip

  \textsuperscript{1}West Chester University, \href{mailto:mfisher@wcupa.edu}{mfisher@wcupa.edu} \par
          \textsuperscript{2}Dalhousie University, \href{mailto:r.nowakowski@dal.ca}{r.nowakowski@dal.ca} \par
    \textsuperscript{3}Center for Functional Analysis, Linear Structures and Applications,\\ University of Lisbon \& ISEL--IPL, \href{mailto:cmfsantos@fc.ul.pt}{cmfsantos@fc.ul.pt}
\end{center}

\begin{abstract}
We present a characterization of the invertible elements of the mis\`{e}re\\ dicotic universe.
\end{abstract}

\section{Introduction}

Combinatorial game theory studies perfect information games in which there are no chance devices (e.g.\ dice) and two players take turns moving alternately.\linebreak Using standard notation, where Left (female) and Right (male) are the players, a\linebreak position is written in the form $G=\{\GL \! \mid \! \GR \}$, where $\GL=\{G^{L_1},G^{L_2},\ldots\}$ is the set of left options from $G$ and $G^{L}$ is a particular left option (and the same for $\GR$ and $G^R$). Normal-play convention states that the last player able to move is the winner; by contrast, mis\`{e}re-play convention states that the last player is the loser. Often, games decompose into components during the play, and, in those situations, a player has to choose a component in which to play -- that motivates the concept of \emph{disjunctive} sum. Also, sometimes, replacing a component $H$ by a component $G$ never hurts a player, no matter what the context is; it is also possible to have a component $G$ acting like $H$ in any context -- these situations motivate a\emph{ partial order} and an \emph{equality} in the structure of games. Disjunctive sum, partial order and equality will be better explained in the next subsection. In a \emph{dicotic} game $G$, every sub-position $G'$ has the property ${G'}^{\mathcal{L}}=\emptyset$ iff ${G'}^{\mathcal{R}}=\emptyset$. Here we are concerned with short dicotic games under mis\`{e}re-play convention (short games are games with finitely many distinct subpositions and no infinite run).\\

\noindent
Normal-play convention is a very special case. Combinatorial games played with normal-play convention, together with the
disjunctive sum, constitutes a group structure \cite{AlberNW2007,BerleCG1982,Con1976,Siegel2013}. The inverse of $G$ is its \emph{conjugate}, obtained recursively by $\sim G=\{\sim G^\mathcal{R}\,|\,\sim G^\mathcal{L}\}$. To check if $G\succcurlyeq H\Leftrightarrow G+(\sim H)\succcurlyeq0$, it is only needed to play $G+(\sim H)$ and see if Left wins going
second. Also, in game practice, a component $G+(\sim G)$ can be removed from the analysis as it behaves like an empty zone of the board. These facts show the importance of invertibility. Regarding normal-play convention, all components are invertible.\\

\noindent
$\sim G$ is obtained from the game $G$ by reversing the roles of Left and Right -- ``turning the board around or switching colours''. In normal play, $G+(\sim G)=0$ which
is decidedly not true in mis\`{e}re-play. Even kowing that, for ease, from now on, we will write $-G$ instead of $\sim G$.\\

\noindent
 Allen asks, regarding the mis\`{e}re dicot structure, when it is true that $G-G=0$ (for\linebreak example, $*+*= 0$) \cite{Nowakowski2019}. McKay, Milley and Nowakowski show that this is true if $G'-G'$ is a next player win for every $G'$ subposition of $G$ (including $G$) \cite{McKay2016}.\\

\noindent
In this paper, we prove the converse implication (much harder to prove),\linebreak establishing that, being $G$ a dicot canonical form, $G$ is invertible if and only if there is no $G'$, subposition of $G$, such that $G'-G'$ is a previous player win.\linebreak Therefore,\textbf{ the main result of this paper is a complete characterization of the invertible elements of the mis\`{e}re dicot structure}.
This document is self\linebreak contained; see \cite{AlberNW2007,BerleCG1982,Con1976,Siegel2013} for more information. Readers fluent in Combinatorial Game Theory may wish to proceed to Section \ref{sec2}.\\

\vspace{-0.5cm}
\subsection{Background on relevant combinatorial game theory}

\noindent
Considering $G=\{\GL \! \mid \! \GR \}$, any position which can be reached from $G$ is called a {\em follower} of $G$ ($G$ itself is a follower of $G$). The possible {\em outcomes} of a position are \linebreak $\mathscr{L}>(\mathscr{P}\|\mathscr{N})>\mathscr{R}$: $\mathscr{L}$eft wins, regardless of moving first or second; $\mathscr{R}$ight wins, regardless of moving first or second; $\mathscr{N}$ext player wins regardless of whether this is Left or Right; $\mathscr{P}$revious player wins regardless of whether this is Left or Right.  The outcome function $o(G)$ will be used to denote the outcome of $G$. The {\em outcome classes} $\L, \N, \R, \P$ are the sets of all games with the indicated outcome, so that we can write $G\in \L$ when $o(G)=\mathscr{L}$. \\

\noindent
Often, games decompose into components during the play. For those situations, the disjunctive sum is formalized: $G+H=\{G^\mathcal{L}+H,G+H^\mathcal{L}\,|\,G^\mathcal{R}+H,G+H^\mathcal{R}\}$.\\

\noindent
The relations {\em inequality} and {\em equivalence} of games are  defined by
$$G \succcurlyeq H  \textrm{ if and only if } o(G+X)\geqslant o(H+X) \textrm{ for all games } X;$$
$$G\equiv H \textrm{ if and only if } o(G+X)= o(H+X) \textrm{ for all games } X.$$
The first means that replacing $H$ by $G$ can never hurt Left, no matter what the context is; the second means that $G$ acts like $H$ in any context.  In this paper, for ease, we always use the symbol $=$; different situations determine if the symbol is being used for games or outcomes. Also, we are using the same symbols for different game conventions. \\

\subsubsection{Mis\`{e}re-play: dicotic forms}

\noindent
Among many other things, the fact that mis\`ere structures lose the group structure makes general mis\`ere analysis very difficult --- \cite{MilleyRenault2017} for a survey. A breakthrough in the study of mis\`ere games occurred when Plambeck and Siegel (\cite{Plamb2005, PlambS2008}) suggested weakened  equality and inequality relations in order to compare games only within a particular universe, being the concept of universe defined in the following way:

\begin{definition} \label{def:universe}
A {\em universe} is a  class  of positions satisfying the following properties:
	\begin{enumerate}
	\item options closure: if $G\in \U$ and $G'$ is an option of $G$ then $G'\in \U$;
	\item disjunctive sum closure: if $G,H\in \U$ then $G+H\in \U$;
	\item conjugate closure: if $G\in \U$ then $-G\in \U.$
	\end{enumerate}
\end{definition}

\noindent
With this concept, it is possible to say that two dicotic positions are equivalent ``modulo dicots'', even if they are different in the full  mis\`ere structure. The \textit{restricted} relations are defined below.

\begin{definition}
\label{def:eq} \cite{PlambS2008}
For a universe $\mathcal{U}$ and games $G,H$, the terms
{\em equivalence} and {\em inequality}, {\em modulo} $\mathcal{U}$, are  defined by
$$G=_{\U} H \textrm{ if and only if } o(G+X)= o(H+X) \textrm{ for all games } X \in \mathcal{U},$$
$$G \succcurlyeq_{\U} H  \textrm{ if and only if } o(G+X)\geqslant o(H+X) \textrm{ for all games } X \in \mathcal{U}.$$
\end{definition}

\noindent
Recently, some advances have been made with regard to the mis\`ere dicotic universe \cite{Dorb2013,LarssonNS2016A,LarssonNPS}, as far as unsubordinated comparison (comparison of $G$ with $H$ only with the forms $G$ and $H$), canonical forms, and conjugate property are concerned, which are presented below.  First, unsubordinated comparison for mis\`ere dicotic universe ($\mathcal{D}^-$):

\begin{theorem}[Unsubordinated order of dicotic mis\`{e}re universe] \label{unsobordinatedorder} $G\succcurlyeq_{\mathcal{D}^-} H$ iff\\

\vspace{-0.15cm}
\noindent
Proviso: $o(G)\geqslant o(H).$\\

\vspace{-0.15cm}
\noindent
Common Normal Part:
\begin{enumerate}[]
  \item For all $G^R$, there is $H^R$ such that $G^R\succcurlyeq_{\mathcal{D}^-} H^R$ or there is $G^{RL}$ such that \linebreak$G^{RL}\succcurlyeq_{\mathcal{D}^-} H$.\\

  \vspace{-0.6cm}
  \item For all $H^L$, there is $G^L$ such that $G^L\succcurlyeq_{\mathcal{D}^-} H^L$ or there is $H^{LR}$ such that\linebreak $G\succcurlyeq_{\mathcal{D}^-} H^{LR}$.
\end{enumerate}
\end{theorem}

\noindent
As usual, from now on, due to the fact that we are only concerned with dicotic mis\`{e}re universe, we will use $=$ and $\succcurlyeq$ instead of $=_{\mathcal{D}^-}$ and $\succcurlyeq_{\mathcal{D}^-}$.

\begin{corollary} \label{ettinger}$G\succcurlyeq 0$ iff $o(G)\geqslant \mathcal{N}$ and for all $G^R$ there is $G^{RL}\succcurlyeq 0$.
\end{corollary}

\begin{corollary} \label{ettinger2}$G=0$ iff $o(G)=\mathcal{N}$, for all $G^R$ there is $G^{RL}\succcurlyeq 0$, and for all $G^L$ there is $G^{LR}\preccurlyeq 0$.
\end{corollary}

\begin{corollary} \label{starplusstar}$*+*=\{*\,|\,*\}=0$.
\end{corollary}

\noindent
Second, reductions and canonical forms:

\begin{theorem}[Domination]\label{thm:domination}
Let $G=\{\GL\mid \GR\}$ be a dicotic form. If  $A, B\in \GL$ and $A\pr B$ then $G = \{\GL\setminus\{A\}\mid \GR\}$.
\end{theorem}

\begin{definition}
For a game $G$ in any universe $\U$, suppose there are followers $A\in \GL$ and
$B\in A^\mathcal R$ with $B\pr_\U G$.
 Then the Left option $A$ is \textit{reversible}, and sometimes, to be specific, $A$
 is said to be \textit{reversible  through} its right option $B$. In addition, $B$ is called a \textit{reversing} option for $A$ and,
 if $B^\mathcal{L}$ is non-empty then $B^\mathcal{L}$
  is a \textit{replacement set} for $A$. In this case,
 $A$ is said to be \textit{non-atomic-reversible}.
 If the reversing option is left-atomic, that is, if $B^\mathcal{L}=\emptyset$, then $A$ is
 said to be \textit{atomic-reversible}.
 \end{definition}

\begin{theorem}[Non-atomic reversibility]\label{thm:nonatomic}
Let $G$ be a dicotic form and suppose that $A$ is a left option of $G$ reversible through $B$.
If  $B^\mathcal{L}$ is non-empty, then\linebreak $G= \left\{(\GL\setminus\{A\})\cup B^\mathcal{L} \mid {\GR}\right\}$.
\end{theorem}

\begin{theorem} [Atomic reversibility] \label{thm:atomicd}
Let $G$ be a dicotic form suppose that $A\in\GL$ is reversible through $B=0$.
\begin{enumerate}
\item If, in $G$, there is a Left winning move $C\in G^{\cal L}\setminus\{A\}$, then
$G= \left\{ \GL\setminus \{A\}\mid {\GR}\right\}$;
\item If $A$ is the only winning Left move in $G$, then $G= \left\{ *,\GL\setminus\{A\}\mid {\GR}\right\}$.
\end{enumerate}
\end{theorem}

\begin{theorem} [Substitution Theorem] \label{thm:substituted}
If $G=\{A\mid C\}$ where $A$ and $C$ are atomic-reversible options
 then $G=0$.
\end{theorem}

\noindent
A form  $G$ is said to be in \textit{canonical form} if none of the previous theorems can be applied to
$G$ or followers to obtain an equivalent game in mis\`ere dicotic universe with different sets of options. In \cite{Dorb2013}, we can find a proof for unicity and simplicity of mis\`ere dicotic canonical forms.\\

\noindent
Third, the conjugate property:

\begin{theorem}\label{Conjugate} For all dicotic forms $G,H$, if $G+H=0$ then $H=-G$.\\
\end{theorem}

\noindent
We recall also the concept of $G^\circ$, \emph{adjoint} of $G$. It is known that $G+G^\circ\in\P$ \cite{Siegel2013}.

\begin{definition}
$G^\circ=\left\{\begin{array}{ll}
                  * & \text{if }\GL=\emptyset \text{ and }\GR=\emptyset; \\
                  \{(\GR)^\circ\,|\,0\} & \text{if }\GL=\emptyset \text{ and }\GR\neq\emptyset;\\
                  \{0\,|\,(\GL)^\circ\} & \text{if }\GL\neq\emptyset \text{ and }\GR=\emptyset; \\
                  \{(\GR)^\circ\,|\,(\GL)^\circ\} & \text{if }\GL\neq\emptyset \text{ and }\GR\neq\emptyset.
                \end{array}
\right.$
\end{definition}

\subsubsection{Absolute facts}

\noindent
The following results hold both in mis\`ere and normal universes. The proofs are not specially difficult and can be found, for example, in \cite{LarssonNS2016A}.

\begin{theorem}\label{lem:oneadded}
For any universe $\U$ and any games $G, H, J \in \U$, if $G\succcurlyeq_\mathcal{U} H$ then
$G+J\succcurlyeq_\mathcal{U} H+J$.
\end{theorem}

\begin{theorem}\label{lem:dis}
Let $G,H\in\mathcal{U}$ and let $J\in \mathcal{U}$ be invertible. Then
$G+J\su_\mathcal{U} H+J$ if and only if
 $G\su_\mathcal{U} H$.
\end{theorem}

\begin{theorem}\label{lem:oneadded}
For any universe $\U$ and any games $G, H, J \in \U$, if $G\succ_\mathcal{U} 0$ and $H\succcurlyeq_\mathcal{U} 0$ then
$G+H\succ_\mathcal{U} 0$.
\end{theorem}

\begin{theorem} \em{(Hand-tying Principle).}\label{lem:greediness}
Let $G\in\mathcal{U}$. If  $|\GL|\ge 1$  then for any $A\in\mathcal{U}$,
$ \{ \GL\cup \{A\}\mid {\GR}\}\su_\mathcal{U} G$.\\
\end{theorem}

\vspace{-0.4cm}
\section{Invertible elements in mis\`ere dicotic universe}\label{sec2}

\noindent
It is easy to check that, regarding mis\`ere dicotic universe, $*2$ is not invertible. By conjugate property, if it was invertible, we would have $*2+*2=0$.\linebreak However, $*2+*2\in\P$ and $0\in\N$. So, a natural question arises: ``Is it true that a\linebreak non-invertible element $G$ always satisfies the property $G-G\in\P$?''. The answer is no; for example, $G=\{0\,|\,*2\}$ is in canonical form, $G-G\in\N$, and $G-G\neq 0$.\linebreak Another question is the following: ``Is it true that a non-invertible element $G$ always has $*2$ as a follower?''. The answer is no. Consider $H=\{0,*\,|\,\{*\,|\,0,*\},\{0\,|\,0,*\}\}$ in canonical form. Then, the game $G=\{0\,|H\}$ has not $*2$ as a follower, $G-G\in\N$, and $G-G\neq0$. These questions touch the essence of the problem, but the \linebreak adequate characterization of the invertible elements of mis\`ere dicotic universe is more sophisticated and it is presented in Theorem \ref{invertibles}.

\subsection{Structure of the proof}\label{structure}

Consider $G$ in canonical form. By Corollary \ref{ettinger2}, $G-G=0$ iff $o(G-G)=\mathcal{N}$, for all $(G-G)^R$ there is $(G-G)^{RL}\succcurlyeq 0$, and for all $(G-G)^L$ there is $(G-G)^{LR}\preccurlyeq 0$. The difficult part is to prove that if $G$ is invertible, then there is no $G'$, subposition of $G$, such that $G'-G'\in\mathcal{P}$.\\

\noindent
The proof is an \emph{argumentum ad absurdum}. If we had a simplest invertible\linebreak canonical form $G$ with an option $G^{L_1}$ with a follower $G^{L_1'}-G^{L_1'}\in\P$, it would be\linebreak mandatory to have some $(G^{L_1}-G)^R\preccurlyeq 0$. Being possible to argue that it should be some $G^{L_1}-G^{L_2}\prec 0$, considering $G^{L_2}$, again by Corollary \ref{ettinger2}, it would be\linebreak mandatory to have $(G-G^{L_2})^L\succcurlyeq 0$. Repeating the process, the \emph{argumentum ad absurdum} is based on the existence of an infinite \emph{carrousel}, meaningless in the context of short games:\\

\noindent
$G^{L_1}-G^{L_2}\prec 0$\\
\noindent
$G^{L_3}-G^{L_2}\succ 0$\\
\noindent
$G^{L_3}-G^{L_4}\prec 0$\\
\noindent
$G^{L_5}-G^{L_4}\succ  0$\\
\noindent
(\ldots)\\

\noindent
A crucial detail is related to the following question: regarding $G^{L_3}-G$, why should be mandatory to have $G^{L_3}-G^{L_4}\prec 0$, and not, say, $G^{L_3}-G^{L_1}\prec 0$? The first idea is the following: due to $G^{L_3}-G^{L_2}\succ 0$, and $G^{L_1}-G^{L_2}\prec 0$, we would have $G^{L_3}-G^{L_2}\succ 0$, and $G^{L_2}-G^{L_1}\succ 0$, and, so, $G^{L_3}-G^{L_2}+G^{L_2}-G^{L_1}\succ 0$. Because of that, $G^{L_3}-G^{L_2}+G^{L_2}-G^{L_1}\succ 0 \Leftrightarrow G^{L_3}-G^{L_1}\succ 0$, and $G^{L_3}-G^{L_1}\prec 0$ would be an impossibility. However, there is no group structure, and we cannot argue like that. If $H$ is not invertible, $G+H-H+W\succ 0$ is not equivalent to $G+W\succ 0$! Fortunately, it is possible to prove a weaker result (Lemma \ref{invertiblelemma}) that makes the proof work for the monoid structure: If $G\succ 0$, then $G+H-H\not\prec 0$.

\subsection{Characterization of the invertible elements of $\mathcal{D}^-$}\label{proof}

\begin{lemma}\label{invertiblelemma}
Let $G$ and $H$ be two dicots. If $G\succ 0$, then $G+H-H\not\prec 0$.
\end{lemma}

\begin{proof}
If $H-H=0$, then $G\succ 0$ implies $G+H-H\succ 0$, and, of course, we have $G+H-H\not\prec 0$.\\

\noindent
If $H-H\neq0$ and $H-H\in \P$, then, playing first, Left has a winning move in $H-H+*$ (she removes the star). So, due to  $G\succ 0$, playing first, Left has a winning move in $G+H-H+*$. However, playing first, Left loses $0+*$. Hence, $G+H-H\not\prec 0$.\\

\noindent
If $H-H\neq0$ and $H-H\in \N$, then let $X=\{0\,|\,\{\mathcal{F}^\circ(H-H)\,|\,0\}\}$, where $\mathcal{F}^\circ(H-H)$ is the set of the adjoints of all followers of $H-H$. By Corollary \ref{ettinger2}, the reason for $H-H\neq0$ must be the existence of some $(H-H)^L$ such that there is no $(H-H)^{LR}\preccurlyeq 0$. Let us see that $(H-H)^L+X$ is a Left winning move in $H-H+X$. In fact, if Right answers in $X$, Left replies $(H-H)^L+\left((H-H)^{L}\right)^\circ$ and wins. On the other hand, if Right answers $(H-H)^{LR}+X$, we have two possibilities: 1) if $(H-H)^{LR}\in\L\cup\P$, Left replies $(H-H)^{LR}$ and wins; 2) if $(H-H)^{LR}\in\N\cup\R$, the reason for $(H-H)^{LR}\not\preccurlyeq0$ must be the existence of some $(H-H)^{LRL}$ such that there is no $(H-H)^{LRLR}\preccurlyeq 0$. Left replies $(H-H)^{LRL}+X$ and the process is repeated, thing that can't go on forever (we are considering short games). So, at some point, Right has to fall in the previous cases, and Left wins. Left, playing first has a winning move in $H-H+X$. Therefore,  due to  $G\succ 0$, playing first, Left has a winning move in $G+H-H+X$. However, playing first, Left loses $0+X$. Hence, $G+H-H\not\prec 0$.\\
\end{proof}

\clearpage
\begin{theorem}[Characterization of invertible elements of mis\`ere dicotic universe]\label{invertibles}
Let $G$ be a dicot in canonical form. Then, $G$ is invertible if and only if there is no $G'$, follower of $G$, such that $G'-G'\in\P$.
\end{theorem}

\begin{proof}
(:$\Leftarrow$) Suppose that there is no $G'$, follower of $G$, such that $G'-G'\in\P$. In particular, due to the fact that $G$ is a follower of $G$, $G-G\in\N$. On the other hand, without loss of generality, against $G^L-G$, Right can reply $G^L-G^L$. Because $G^L$ is in the conditions of the theorem, by induction, $G^L$ is invertible and, consequently, by conjugate property, $G^L-G^L=0$. The proviso and the common normal part of the Corollary \ref{ettinger2} are satisfied and, so, $G-G=0$, which means that $G$ is invertible.\\

\noindent
(:$\Rightarrow$) Suppose that there is $G$ invertible with a follower $G'$ such that $G'-G'\in\P$. Assume that $G$ is a simplest form in those conditions. By conjugate property, we have $G-G=0$ and, due to that, the common normal part must be satisfied. Because there is a follower $G'$ in the conditions of the theorem, $G$ is not $\{\,|\,\}$, and there are moves in $G-G=0$.\\

\noindent
Without loss of generality, suppose $G^{L_1}$ has a follower in the conditions of the\linebreak theorem, and consider the Left option $G^{L_1}-G$. By Corollary \ref{ettinger2}, Right must have an answer less or equal than zero.\\

\noindent
 $G^{L_1}-G^{L_1}$ cannot be $0$ because $G$ is a simplest invertible element in the conditions of the theorem. Hence, $G^{L_1}$ is not invertible. $G^{L_1}-G^{L_1}$ cannot be less than zero because its tree is symmetric. Therefore, $G^{L_1}-G^{L_1}$ cannot be a Right's answer less or equal than zero.\\

\noindent
On the other hand, we cannot have $G^{L_1R}-G\preccurlyeq 0$. If so, the fact that $G$ is invertible allows us to conclude that $G^{L_1R}\preccurlyeq G$ and, due to Theorem \ref{thm:atomicd} and the fact that $G$ is in canonical form, $G^{L_1}$ must be the atomic reversible option $*$, making $G^{L_1}$ invertible.\\

\noindent
Right's reply must be some $G^{L_1}-G^{L_2}\preccurlyeq 0$. More, we must have $G^{L_1}-G^{L_2}\prec 0$\linebreak because, if  $G^{L_1}-G^{L_2}= 0$, $G^{L_1}$ would be invertible. Observe that $G^{L_2}$ is not\linebreak invertible, otherwise $G^{L_1}\prec G^{L_2}$ and we would have a dominated option in a\linebreak canonical form. That is a first fact.\\

\noindent
Consider now $G-G^{L_2}$, a right option of $G-G$. Using similar arguments, there exists $G^{L_3}-G^{L_2}\succ 0$ ($G^{L_3}$ not invertible). That is a second fact.\\

\clearpage
\noindent
Against $G^{L_3}-G$, a left option of $G-G$, we cannot have $G^{L_3}-G^{L_1}\prec 0$. If so, $G^{L_3}-G^{L_1}+G^{L_1}-G^{L_2}\prec 0$, contradicting Lemma \ref{invertiblelemma}. In fact, against
$G^{L_k}-G$, a left option of $G-G$, we cannot have $G^{L_k}-G^{L_{k-i}}\prec 0$. If so, $$G^{L_k}-G^{L_{k-i}}+G^{L_{k-i}}-G^{L_{k-i+1}}+\ldots-G^{L_{k-1}}\prec 0,$$ contradicting Lemma \ref{invertiblelemma}.\\

\noindent
There exists $G^{L_3}-G^{L_4}\succ 0$ ($G^{L_4}$ not invertible). That is a third fact.\\

\noindent
Repeating the process, the existence of the following options is mandatory:

 $$G^{L_1}-G^{L_2}\prec 0,\,G^{L_3}-G^{L_2}\succ 0,\,G^{L_3}-G^{L_4}\prec 0,\,G^{L_5}-G^{L_4}\succ 0,\ldots$$

\noindent
But, due to the fact that we are considering short games, such an infinite sequence cannot exist. So, there is no invertible dicot $G$ in canonical form with a follower $G'$ such that $G'-G'\in\P$.
\end{proof}

\vspace{0.8cm}
\begin{observation}
Observe that the result works only for canonical forms. For example, $\{0,*,*2\,|\,0\}$ is invertible and $*2+*2\in\P$. However, $\{0,*,*2\,|\,0\}$ is not in canonical form; its canonical form is $\{0,*\,|\,0\}$.
\end{observation}

\vspace{0.8cm}
\begin{corollary}
Let $G$ be a dicot in canonical form. Then, if $*2$ is a follower of $G$, $G$ is not invertible.
\end{corollary}

\begin{proof}
Immediate consequence of Theorem \ref{invertibles} and the fact that $*2+*2\in\P$.
\end{proof}

\vspace{0.8cm}
\begin{corollary}
Let $G$ be an invertible dicot in canonical form. Then, all followers of $G$ are invertible.
\end{corollary}

\begin{proof}
Immediate consequence of Theorem \ref{invertibles} and the fact the followers of a follower of $G$ are also followers of $G$.
\end{proof}

\end{document}